\documentclass[11pt, reqno]{amsart}
\usepackage{amsthm,amssymb,amsmath}

\usepackage{graphicx}
\usepackage[foot]{amsaddr}

\newtheorem{theorem}{Theorem}
\newtheorem{proposition}{Proposition}

\newtheorem{estimate}{}

\theoremstyle{definition}

\newtheorem*{remark}{\sc Remark}
\newtheorem*{remarks}{\sc Remarks}
\newtheorem*{example*}{\bf Example}

\newcommand{\loc}{{\rm loc}}

\newcommand{\Real}{{\rm Re}\,}
\newcommand{\Imag}{{\rm Im}\,}

\newcommand{\clos}{{\rm clos}}

\expandafter\def\expandafter\normalsize\expandafter{%
    \normalsize
    \setlength\abovedisplayshortskip{8pt}
    \setlength\belowdisplayshortskip{8pt}
}

\begin{document}

\title[$L^p$-theory of $-\Delta + b\cdot\nabla$, and Feller processes]
{A new approach to the $L^p$-theory of $-\Delta + b\cdot\nabla$, and its \\ applications to Feller processes \\ with general drifts}

\author{Damir Kinzebulatov}

\address{{\scriptsize Department of Mathematics, University of Toronto, 40 St.~George Str., Toronto, ON, M5S2E4, Canada}}

\email{damir.kinzebulatov@utoronto.ca}



\begin{abstract}
We develop a detailed regularity theory of $-\Delta +b \cdot \nabla$ in $L^p(\mathbb R^d)$, for  a wide class of vector fields. 
The $L^p$-theory allows us to construct associated strong Feller process in $C_\infty(\mathbb R^d)$.
Our starting object is an 
operator-valued function, 
 which, we prove, 
coincides with the resolvent of an operator realization of $-\Delta + b\cdot \nabla$, the generator of a holomorphic $C_0$-semigroup on $L^p(\mathbb R^d)$. Then the very form of the operator-valued function yields crucial information about smoothness of the
domain of the generator. 
\end{abstract}

\maketitle


\subsection{}
Let $\mathcal L^d$ be the Lebesgue measure on $\mathbb R^d$, $L^p=L^p(\mathbb R^d,\mathcal L^d)$ 
and $W^{1,p}=W^{1,p}(\mathbb R^d,\mathcal L^d)$ 
the standard Lebesgue 
and Sobolev 
spaces, $C^{0,\gamma}=C^{0,\gamma}(\mathbb R^d)$ the space of H\"{o}lder continuous functions ($0<\gamma<1$), $C_b=C_b(\mathbb R^d)$ the space of bounded continuous functions endowed with the $\sup$-norm, 
$C_\infty \subset C_b$ the closed subspace of functions vanishing at infinity,
$\mathcal W^{\alpha,p}$, $\alpha>0$, the Bessel  space endowed with norm $\|u\|_{p,\alpha}:=\|g\|_p$,  
$u=(1-\Delta)^{-\frac{\alpha}{2}}g$, $g \in L^p$, and $\mathcal W^{-\alpha,p}$ the dual of $\mathcal W^{\alpha,p}$.
We denote by $\mathcal B(X,Y)$ the space of bounded linear operators between complex Banach spaces $X \rightarrow Y$, endowed with operator norm $\|\cdot\|_{X \rightarrow Y}$;  $\mathcal B(X):=\mathcal B(X,X)$. Set $\|\cdot\|_{p \rightarrow q}:=\|\cdot\|_{L^p \rightarrow L^q}$.


For each $p \geqslant 1$ and $p'=p/(p-1)$, 
by $\langle u,v\rangle$ we denote the $( L^p,L^{p'})$ pairing, so that
$$
\langle u,v\rangle = \langle u\bar{v}\rangle :=\int_{\mathbb R^d}u\bar{v}d\mathcal L^d \qquad (u \in L^p, v \in L^{p'}).
$$

\subsection{}
Consider the following classes of vector fields.



\medskip

(1) We say that a $b:\mathbb R^d \rightarrow \mathbb C^d$ belongs to $\mathbf{F}_\delta$, the class of form-bounded vector fields, and write $b  \in \mathbf{F}_\delta$, if 
$b$ is $\mathcal{L}^d$-measurable, and there exists $\lambda = \lambda_\delta > 0$ such that 
$$
\| b (\lambda - \Delta )^{-\frac{1}{2}} \|_{2 \to 2} \leqslant \sqrt{\delta}. 
$$

(2) We say that a $b:\mathbb R^d \rightarrow \mathbb C^d$ belongs to the Kato class $\mathbf{K}^{d+1}_\delta$, and write $ b \in \mathbf{K}^{d+1}_\delta$, if  $b$ is $\mathcal{L}^d$-measurable, and there exists $\lambda = \lambda_\delta > 0$ such that
$$
\| b (\lambda - \Delta)^{-\frac{1}{2}} \|_{1 \to 1} \leqslant \delta.
$$


\smallskip

(3)  We say that a $b: \mathbb{R}^d \to  \mathbb{C}^d$ belongs to $\mathbf{F}_\delta^{\scriptscriptstyle \frac{1}{2}}$, the class of \textit{weakly} form-bounded vector fields, and write $b \in \mathbf{F}_\delta^{\scriptscriptstyle \frac{1}{2}}$, if 
$b$ is $\mathcal{L}^d$-measurable, and there exists $\lambda = \lambda_\delta > 0$ such that
$$
\| |b|^\frac{1}{2} (\lambda - \Delta)^{-\frac{1}{4}} \|_{2 \to 2} \leqslant \sqrt{\delta}.
$$



$\text{Simple examples show:}\quad \mathbf{F}_{\delta_1} - \mathbf{K}^{d+1}_\delta \neq \varnothing$
$
\text { and } \mathbf{K}^{d+1}_{\delta_1} - \mathbf{F}_\delta \neq \varnothing \quad \text{for any $\delta, \delta_1 > 0 $;}
$
\medskip
$$\mathbf{K}_\delta^{d+1} \subsetneq \mathbf{F}_\delta^{\scriptscriptstyle \frac{1}{2}}, \qquad 
\mathbf{F}_{\delta_1} \subsetneq \mathbf{F}_\delta^{\scriptscriptstyle \frac{1}{2}}\; \text{ for  
$\delta = \sqrt{\delta_1} $} 
\text{ by Heinz inequality \cite{He}} \footnote{In fact, at least for $b \in L^{d,\infty}$ ($d\geqslant 3$), $\delta<\sqrt{\delta_1}$, see \cite[Corollary 2.9]{KPS}.};$$ 
$$\biggl( b \in  \mathbf{F}_{\delta_1}^{~} \text{ and } \mathsf{f} \in \mathbf{K}^{d+1}_{\delta_2} \biggr) \Longrightarrow \biggl( b + \mathsf{f} \in \mathbf{F}^\frac{1}{2}_{\delta}, \; \sqrt{\delta} = \sqrt[4]{\delta_1} + \sqrt{\delta_2} \biggr).
$$

\medskip

Denote
$$
m_d := \pi^{\frac{1}{2}} (2e)^{-\frac{1}{2}} d^\frac{d}{2} (d-1)^{\frac{1-d}{2}}, \qquad c_p:= pp'/4. 
$$ 
The main results of our paper:

\begin{theorem}[$L^p$-theory]
\label{thm1}

Let $d \geqslant 3$ and $b:\mathbb{R}^d \to \mathbb{C}^d.$ Assume that  $ b \in \mathbf{F}^\frac{1}{2}_\delta$, $m_d \delta < 1.$ Then, for every $p \in \mathcal I:=\bigl(\frac{2}{1+\sqrt{1-m_d\delta}},\frac{2}{1-\sqrt{1-m_d\delta}}\bigr),$  there exists a $C_0$-semigroup $e^{-t \Lambda_p(b)}$ in $L^p$ such that

{\rm (\textit{i})}\; The resolvent set $\rho(-\Lambda_p(b))$ contains the half-plane $\mathcal O:=\{\zeta \in \mathbb C: \Real\,\zeta \geqslant \kappa_d\lambda_\delta\}$, $\kappa_d := \frac{d}{d-1}$, and the resolvent admits the representation:
$$
(\zeta + \Lambda_p(b))^{-1} = \Theta_p(\zeta, b), \quad \zeta \in \mathcal O,
$$
where
\begin{equation}
\label{R_p}
\Theta_p(\zeta, b):= (\zeta - \Delta)^{-1} - Q_p (1 + T_p)^{-1} G_p, 
\end{equation}
the operators $Q_p, G_p, T_p \in \mathcal B(L^p)$, 
\begin{align*}
\|G_p \|_{p \to p} \leqslant C_1 | \zeta |^{-\frac{1}{2 p^\prime}}, \; \|Q_p\|_{p} \leqslant C_2 | \zeta |^{-\frac{1}{2}-\frac{1}{2 p}}, \; \| T_p\|_{p} \leqslant m_d c_p \delta < 1,
\end{align*}
$$
\;G_p \equiv G_p(\zeta, b) := b^\frac{1}{p} \cdot \nabla (\zeta - \Delta)^{-1}, \quad b^{\frac{1}{p}}:=|b|^{\frac{1}{p}-1}b,
$$
$Q_p$, $T_p$ are the extensions by continuity of densely defined (on $\mathcal{E} :=  \bigcup_{\epsilon >0} e^{-\epsilon|b|} L^p$) operators
\begin{align*}
& Q_p |_{\mathcal E}\equiv Q_p(\zeta, b)|_{\mathcal E}: = (\zeta -\Delta)^{-1} |b|^{\frac{1}{p^\prime}}, \;
& T_p|_{\mathcal E} \equiv T_p(\zeta, b)|_{\mathcal E}:= b^\frac{1}{p} \cdot \nabla (\zeta -\Delta)^{-1} |b|^\frac{1}{p^\prime}.
\end{align*}

\medskip

{\rm (\textit{ii})} \; $e^{-t\Lambda_p(b)}$ is holomorphic: there is a constant $C_p$ such that
\[
\|(\zeta + \Lambda_p(b))^{-1} \|_{p \to p} \leqslant C_p |\zeta|^{-1}, \quad \zeta \in \mathcal O.
\]

{\rm (\textit{iii})} \; For each $1 \leqslant r<p<q$ and $\zeta \in \mathcal O,$ define  
\begin{align*}
&G_{p}(r) \equiv G_{p}(r,\zeta,b) := b^\frac{1}{p} \cdot \nabla (\zeta - \Delta)^{-\frac{1}{2} - \frac{1}{2r}}, \quad  G_{p}(r) \in \mathcal B(L^p),\\
&Q_{p}(q) \equiv Q_{p}(q,\zeta, b) := (\zeta -\Delta)^{- \frac{1}{2q'}} |b|^{\frac{1}{p'}} \text{ on } \mathcal E.
\end{align*}
The extension of $Q_p(q)$ by continuity  we denote again by $Q_{p}(q)$.
Then
\begin{align*}
& \Theta_p(\zeta, b)= (\zeta - \Delta)^{-1} - (\zeta - \Delta)^{-\frac{1}{2}-\frac{1}{2q}} Q_{p}(q) (1 + T_p)^{-1} G_{p}(r) (\zeta - \Delta)^{-\frac{1}{2r'}}, \quad \zeta \in \mathcal O;\\
\notag
&\Theta_p(\zeta,b) \text{ extends by continuity to an operator in }
 \mathcal B\bigl(\,\mathcal W^{-\frac{1}{r'},p}, \; \mathcal W^{1+\frac{1}{q},p}\,\bigr).
\end{align*}

{\rm (\textit{iv})}\; 
$D\bigl(\Lambda_p(b)\bigr)\subset \mathcal W^{1+\frac{1}{q},p}$ ($q>p$). In particular, if $m_d \delta <4\frac{d-2}{(d-1)^2}$, there exists $p \in \mathcal I$, $p>d-1$,  so $D\bigl(\Lambda_p(b)\bigr) \subset C^{0,\gamma}$, $\gamma<1-\frac{d-1}{p}$.

\medskip

{\rm (\textit{v})} \;Let $u \in D(\Lambda_p (b))$. Then 
\begin{align*}
& \langle \Lambda_p (b) u, v \rangle = \langle u, -\Delta v \rangle + \langle b\cdot\nabla u, v \rangle, \quad \;v \in C_c^\infty(\mathbb{R}^d); \\
& u \in \mathcal W^{2,1}_{\loc}.
\end{align*}

{\rm (\textit{vi})} \; $e^{-t \Lambda_p (b_n)} \overset{s}{\rightarrow} e^{-t \Lambda_p(b)} \text{ in $L^p$}$, \quad $t>0,$

\medskip
\noindent where $b_n:=
b$ if $|b| \leqslant n$,
$b_n:=n|b|^{-1}b$ if $|b|>n$, and
 $\Lambda_{p}(b_n):=-\Delta + b_n\cdot \nabla$, $D(\Lambda_p(b_n))=W^{2,p}$.

\medskip

{\rm (\textit{vii})} \; If $b$ is real-valued, then $e^{-t\Lambda_p(b)}$ is positivity preserving.

\medskip

{\rm (\textit{viii})} \; $\|e^{-t \Lambda_p(b)} \|_{p \to r} \leqslant c_{p,r} t^{-\frac{d}{2}(\frac{1}{p}-\frac{1}{r})}, \; 0 < t \leqslant 1, \; p < r.$

\end{theorem}

Let 
$$
\eta(x):=\left\{
\begin{array}{ll}
c\exp\left(\frac{1}{|x|^2-1}\right)& \text{ if } |x|<1, \\
0, & \text{ if } |x| \geqslant 1,
\end{array}
\right.
$$
where $c$ is adjusted to $\int_{\mathbb R^d} \eta(x)dx=1$. Define the standard mollifier  $$\eta_\varepsilon(x):=\frac{1}{\varepsilon^{d}}\eta\left(\frac{x}{\varepsilon}\right), \quad  \varepsilon>0, \;\;x\in \mathbb R^d.$$

\begin{theorem}[$C_\infty$-theory]
\label{thm2}
Let $d \geqslant 3$. Assume that $$b:\mathbb R^d \rightarrow \mathbb R^d, \quad
b \in \mathbf{F}_\delta^{\scriptscriptstyle \frac{1}{2}}, \quad 
m_d \delta <4\frac{d-2}{(d-1)^2}.$$ 
Then for every $\tilde{\delta}>\delta$ satisfying $m_d \tilde{\delta}<4\frac{d-2}{(d-1)^2}$ there exist 
$\varepsilon_n>0$, $\varepsilon_n\downarrow 0$, such that 
$$
\tilde{b}_{n}:=\eta_{\varepsilon_n} \ast b_n \in C^\infty(\mathbb R^d,\mathbb R^d) \cap \mathbf{F}_{\tilde{\delta}}^{\scriptscriptstyle \frac{1}{2}}, \quad n=1,2,\dots,
$$
and

\smallskip

{\rm(\textit{i})}
$$e^{-t\Lambda_{C_\infty}(b)}:=\text{\small $s\text{-}C_\infty\text{-}$}\lim_n e^{-t\Lambda_{C_\infty}(\tilde{b}_n)}, \quad t>0,
$$
determines a positivity preserving contraction $C_0$-semigroup on $C_\infty$, \\ where $\Lambda_{C_\infty}(\tilde{b}_n):=-\Delta+\tilde{b}_n\cdot \nabla$, $D(\Lambda_{C_\infty}(\tilde{b}_n))=C^2 \cap C_\infty$.

\medskip

{\rm(\textit{ii})}~(Strong Feller property)~$(\mu+\Lambda_{C_\infty}(b))^{-1}[L^p \cap C_\infty] \subset  C^{0,\alpha}$,  
\smallskip

\noindent $\mu>0$, $p \in \bigl(d-1,\frac{2}{1-\sqrt{1-m_d\delta}}\bigr)$, $\alpha < 1-\frac{d-1}{p}$.

\smallskip

{\rm(\textit{iii})}~The
integral kernel $e^{-t\Lambda_{C_\infty}(b)}(x,y)$ {\rm(}$x,y \in \mathbb R^d${\rm)} of $e^{-t\Lambda_{C_\infty}(b)}$ determines the
(sub-Markov) transition probability function of a strong Feller process.
\end{theorem}

\begin{remarks}

1.~Theorem \ref{thm1} allows us to move the problem 
of convergence in $C_\infty$ (in Theorem \ref{thm2}(\textit{i}))
to $L^p$, a space having much weaker topology (locally).

2.~In place of $\tilde{b}_n$ in Theorem \ref{thm2}, one can take (without changing the proof) $$\hat{b}_n:=\eta_{\varepsilon_n} \ast (\mathbf{1}_nb) \in
C_0^\infty(\mathbb R^d,\mathbb R^d) \cap \mathbf{F}_{\tilde{\delta}}^{\scriptscriptstyle \frac{1}{2}}, \quad m_d \tilde{\delta} <4\frac{d-2}{(d-1)^2}, \quad n=1,2,\dots,
$$
where 
$\mathbf{1}_n$
is the characteristic function of $\{x \in \mathbb R^d: |b(x)| \leqslant m_n, |x| \leqslant n\}$, $m_n<n$ for all $n$,  $m_n \uparrow \infty$ as $n \uparrow \infty$.


\end{remarks}

\subsection{On the existing results prior to our work.}
\label{sect_existing}

First, it had been known for a long time, see [KS], that, for $b: \mathbb{R}^d \to \mathbb{R}^d$, $d \geqslant 3$, and $b \in \mathbf{F}_\delta,$

(i) \textit{(The basic fact)} $D(\Lambda_p(b)) \subset W^{1, j p}$ \textit{for every} $p \in (d-2, 2/\sqrt{\delta}), \; j=\frac{d}{d-2},$ provided that $0<\delta<\min\{1, (\frac{2}{d-2})^2\}.$

(ii) \textit{If, in addition to the assumptions in (i), $| b | \in L^2 + L^\infty,$ then 
$$\text{{\small $s$-$C_\infty$-}}\lim_n e^{-t \Lambda_{C_\infty} (b_n)}$$
exists uniformly in each finite interval of $t\geqslant 0,$ and hence determines a strongly Feller semigroup on} $C_\infty.$

\begin{remark} The additional (to $| b | \in L^2_\loc$) assumption $| b | \in L^2 + L^\infty$  in (ii)  was removed in \cite{Ki}.
\end{remark}

\begin{theorem}[Yu.~A.~Semenov]
\label{thm3}
Let $b: \mathbb{R}^d \to  \mathbb{R}^d, \;d \geqslant 3.$ 

$\mathbf{a})$ If $b \in \mathbf{K}^{d+1}_\delta, \;m_d \delta < 1$, then, for each $p \in [1, \infty),$ $s$-$L^p$-$\lim_n e^{-t \Lambda_p(b_n)}$ exists uniformly on each finite interval of $t \geqslant 0,$ and hence determines a $C_0$-semigroup $e^{-t \Lambda_p(b)}.$

$e^{-t \Lambda_p(b)}$ is a quasi-bounded positivity preserving $L^\infty$-contraction $C_0$- semigroup;
\[
\|e^{-t \Lambda_r(b)}\|_{r \to q} \leqslant c_{d, \delta}\; t^{-\frac{d}{2}(\frac{1}{r}-\frac{1}{q})}  \text{ for all } 0 <t \leqslant 1,\; 1 \leqslant r < q \leqslant \infty;
\]
The resolvent set $\rho(-\Lambda_p(b))$ contains the half-plane $\mathcal O$,
\[
(\zeta + \Lambda_p(b))^{-1} = \Theta_p(\zeta, b), \;\; \zeta \in \mathcal O,
\]
\begin{align*}
& \Theta_p(\zeta, b):= (\zeta - \Delta)^{-1} - (\zeta-\Delta)^{-\frac{1}{2}}S_p (1 + T_p)^{-1} G_p, \\
&S_p:=(\zeta-\Delta)^{-\frac{1}{2}}|b|^{\frac{1}{p'}}, \quad G_p:=b^{\frac{1}{p}}\cdot \nabla (\zeta-\Delta)^{-1}, \quad T_p:=b^{\frac{1}{p}}\cdot \nabla (\zeta-\Delta)^{-1}|b|^{\frac{1}{p'}};\\
& \Theta_p(\zeta,b) \in \mathcal B(L^p,\mathcal W^{1,p}); \\
& D(\Lambda_p(b)) \subset \mathcal{W}^{1, \;p}. \text{ In particular, for }p > d, \; D(\Lambda_p(b)) \subset C^{0, \;\alpha},\; \alpha = 1 -\frac{d}{p};\\
& \langle \Lambda_p (b) f, g \rangle = \langle \nabla f, \nabla g \rangle + \langle b\cdot\nabla f, g \rangle, \quad \quad f \in D(\Lambda_p (b)), \;g \in C_c^\infty(\mathbb{R}^d).
\end{align*}

$\mathbf{b})$ If $b \in \mathbf{F}_\delta^\frac{1}{2}, \;\delta < 1,$ then, for each $p \in [2, \infty),$ 
{\small $s$-$L^p$-}$\lim_n e^{-t \Lambda_p(b_n)}$ exists uniformly on each finite interval of $t\geqslant 0,$ and hence determines a $C_0$-semigroup $e^{-t \Lambda_p(b)}.$ 

\smallskip

$e^{-t \Lambda_p(b)}$ is a quasi-bounded positivity preserving $L^\infty$-contraction $C_0$- semigroup.

\smallskip

$\|e^{-t \Lambda_r(b)}\|_{r \to q} \leqslant c_{d, \delta}\; t^{-\frac{d}{2}(\frac{1}{r}-\frac{1}{q})}$  for all $0 <t \leqslant 1,\; 2 \leqslant r < q \leqslant \infty.$

\smallskip

$D(\Lambda_2(b)) \subset W^{\frac{3}{2},2}.$

\smallskip

$\langle \Lambda_2 (b) f, g \rangle = \langle \nabla f, \nabla g \rangle + \langle b\cdot\nabla f, g \rangle, \quad \quad f \in D(\Lambda_2 (b)), \;g \in C_c^\infty(\mathbb{R}^d).$

\end{theorem}

 We outline the proof of Semenov's results.

\begin{proof}
$\mathbf{a})$ Indeed, for all $\zeta$ with $\Real \;\zeta >0,$
 \[
 |\nabla (\zeta - \Delta)^{-1}(x,y)|\leqslant m_d (\kappa_d^{-1} \Real \zeta -\Delta)^{-\frac{1}{2}}(x,y) \text{ pointwise on } \mathbb{R}^d \times \mathbb{R}^d
 \]
 (see \eqref{lem_lambda_est} in the Appendix).
 Therefore, for $b \in \mathbf{K}^{d+1}_\delta,$
 \[ 
 \| b \cdot \nabla (\zeta - \Delta)^{-1} \|_{1 \to 1}\leqslant m_d \delta, \quad\quad \Real \;\zeta \geqslant \kappa_d \lambda,
 \]
and so by the Miyadera perturbation theorem, the operator $- \Lambda_1(b) := \Delta - b \cdot \nabla$ of domain $D(\Lambda_1(b)) = \mathcal{W}^{2,1}$ is the generator of a quasi-bounded $C_0$ semigroup on $L^1$ \textit {whenever} $m_d \delta < 1.$

Clearly 
$b_n \in \mathbf{K}^{d+1}_\delta,$ $\| b_n \cdot \nabla (\zeta - \Delta)^{-1} \|_{1 \to 1}\leqslant m_d \delta,$ and, for $m_d \delta < 1$ and every $f \in D(\Lambda_1(b)),$ $\Lambda_1 (b_n) f  \overset{s} \rightarrow \Lambda_1 (b) f$ by the Dominated Convergence Theorem. (See, if needed, \eqref{A0}.) The latter easily implies the strong resolvent and the semigroup convergence of $\Lambda_1(b_n)$ to $\Lambda_1(b).$

Then, for each $n=1,2,\dots,$ the semigroups $e^{- t \Lambda_1(b_n)}, t>0,$ are positivity preserving $L^\infty$-contractions, and so is $e^{-t \Lambda_1(b)}.$ The bounds
\[
\|e^{-t \Lambda_1(b)} \|_{1 \to 1}\leqslant M e^{t \omega}, \; \omega = \kappa_d\lambda, \text{ and } \|e^{-t \Lambda_1(b)} f \|_\infty\leqslant \|f\|_\infty, \; f \in L^1 \cap L^\infty,
\]
yield via the Riesz interpolation theorem
\[
\|e^{-t \Lambda_1(b)} f \|_p\leqslant M^{1/p} e^{t \omega /p} \|f\|_p, \;\;f \in L^1 \cap L^\infty.
\]
Therefore, we obtain a family $\{e^{-t \Lambda_p(b)} \}_{1\leqslant p < \infty}$ of consistent $C_0$-semigroups by setting $e^{-t \Lambda_p(b)} := $ the extension by continuity in $L^p$ of $e^{-t \Lambda_1(b)} \mid L^1 \cap L^\infty.$

Next, for each $p \in [1, \infty)$ and all $f \in \mathcal{E} := \bigcup_{\epsilon > 0} e^{-\epsilon |b|} L^p,$ the inequality 
\[
 \| |b|^\frac{1}{p} (\lambda -\Delta)^{-\frac{1}{2}} |b|^\frac{1}{p^\prime} f\|_p\leqslant \delta \|f\|_p 
 \]
as well as inequality
\[
\| (|b|+\sqrt{\lambda})^\frac{1}{p} (\lambda -\Delta)^{-\frac{1}{2}} (|b|+\sqrt{\lambda})^\frac{1}{p^\prime} f\|_p\leqslant (1+\delta) \|f\|_p
\]
follow from the very definition of $\mathbf{K}^{d+1}_\delta$ using H\"older's inequality.
Note that the last inequality clearly implies that
\[
\||b|^\frac{1}{p}(\lambda -\Delta)^{-\frac{1}{2}} \|_{p \to p}\leqslant (1+\delta) \lambda^{-\frac{1}{2 p^\prime}},
\]  
and the first inequality implies that, for every $\zeta \in \mathcal O$, $p \in [1, \infty)$ and all $f \in \mathcal{E},$
\[
\|b^\frac{1}{p} \cdot \nabla (\zeta - \Delta)^{-1} \;|b|^\frac{1}{p^\prime} f \|_p\leqslant m_d \| |b|^\frac{1}{p} (\lambda -\Delta)^{-\frac{1}{2}} |b|^\frac{1}{p^\prime} |f| \|_p\leqslant m_d \delta \|f\|_p.
\]

Now, for every $p \in [1, \infty)$ and $\zeta \in \mathcal O$, we define operators $G_p, S_p, T_p$ acting in $ L^p$ by
\[
G_p = b^\frac{1}{p} \cdot \nabla(\zeta - \Delta)^{-1}, \; S_p =(\zeta - \Delta)^{-\frac{1}{2}}| b |^\frac{1}{p^\prime}, \; T_p = b^\frac{1}{p} \cdot \nabla(\zeta -\Delta)^{-1} \;|b|^\frac{1}{p^\prime}.
\]
It is seen that $G_p$ is bounded:
\[
\|G_p \|_{p \to p}\leqslant m_d \|b^\frac{1}{p}(\lambda - \Delta)^{-\frac{1}{2}} \|_{p \to p}\leqslant m_d(1+\delta)\lambda^{-\frac{1}{2 p^\prime}}.
\]
 $S_p$ and $T_p$ are densely defined (on $\mathcal{E}$) and, for all $ f \in \mathcal{E},$
\[
 \| S_p f \|_p\leqslant (1+\delta)^{-1} \lambda^{-\frac{1}{2 p}} \| f \|_p \text{ and } \| T_p f \|_p\leqslant m_d \delta \| f \|_p. 
\] 
Their extensions by continuity we denote again by $S_p, T_p.$

Next, we define an operator function $\Theta_p(\zeta, b)$ in $L^p$ by 
\[ 
\Theta_p(\zeta, b):=  (\zeta -\Delta)^{-1} - (\zeta -\Delta)^{-\frac{1}{2}} S_p\;(1 +T_p)^{-1} \;G_p \quad \quad \zeta \in \mathcal O.
\]
Obviously
\[
\Theta_p(\zeta, b) \in \mathcal{B}(L^p) \text{ and } \Theta_p(\zeta, b) \in \mathcal{B}(L^p, W^{1,p}).
\]
It is also seen that

 $(\zeta + \Lambda_1(b))^{-1} = \Theta_1 (\zeta, b),\;\;(\zeta + \Lambda_p(b))^{-1} \mid L^1 \cap L^p = \Theta_p(\zeta, b) \mid L^1 \cap L^p,$ and so
\[
(\zeta + \Lambda_p(b))^{-1} = \Theta_p(\zeta, b),  \quad \quad \zeta \in \mathcal O.
\]
The latter implies that $D(\Lambda_p(b)) \subset W^{1,p}$, for all $p\in [1,\infty)$. The main assertion is proved.

$\mathbf{b})$ Let $b \in \mathbf{F}_\delta^{\scriptsize \frac{1}{2}}$, $\delta<1$. Define $H=|b|^{\frac{1}{2}}(\zeta-\Delta)^{-\frac{1}{4}}, \; S=b^{\frac{1}{2}}\cdot\nabla(\zeta-\Delta)^{-\frac{3}{4}}$ and
\begin{align*}
\tag{$\ast$}
\label{first_repr_semenov}
\Theta_2(\zeta,b):= & (\zeta-\Delta)^{-\frac{3}{4}}(1 + H^*S)^{-1}(\zeta-\Delta)^{-\frac{1}{4}} \\
= & (\zeta - \Delta)^{-1} - (\zeta - \Delta)^{-\frac{3}{4}} H^* (1 + S H^*)^{-1} S ((\zeta - \Delta)^{-\frac{1}{4}}, \;\; \Real \zeta \geqslant \lambda.
\end{align*}
We have 
$$
\|H^*S\|_{2 \rightarrow 2}\leqslant \| H \|_{2 \rightarrow 2} \|S \|_{2 \rightarrow 2}\leqslant \|H \|_{2 \rightarrow 2}^2 \|\nabla(\zeta-\Delta)^{-\frac{1}{2}}\|_{2 \rightarrow 2} \leqslant \delta \text{ and } \|\Theta_2(\zeta,b)\|_{2 \rightarrow 2} \leqslant (1-\delta)^{-1}|\zeta|^{-1}. 
$$
Note that $ D(\Lambda_2(b_n)) = W^{2, 2}$ and, for all $\Real\; \zeta \geqslant \lambda,$ by the first representation of $\Theta_2(\zeta, b_n),$
\[
\Theta_2(\zeta, b_n)^{-1}|W^{2, 2} = (\zeta +\Lambda_2(b_n))|W^{2, 2}, \quad \Theta_2(\zeta, b_n)= (\zeta +\Lambda_2(b_n))^{-1}, 
\]
\[
\zeta \Theta_2(\zeta,b_n) \overset{s}\rightarrow 1 \text{ as } \zeta \uparrow \infty \text{ by the second representation of } \Theta_2(\zeta, b_n).
\]
\textit{Therefore, $\Theta_2(\zeta, b_n)$ is the resolvent of $-\Lambda_2(b_n).$} 

Since $\|\Theta_2(\zeta, b_n) \|_{2 \to 2} \leqslant (1-\delta)^{-1} |\zeta|^{-1},$ the semigroups $e^{-t \Lambda_2 (b_n)}$ are holomorphic and equi-bounded.

Finally, it is seen that $\Theta_2(\zeta, b_n) \overset{s}\rightarrow \Theta_2(\zeta, b)$ in $L^2$ on $\Real\; \zeta \geqslant \lambda,$ and  $\mu \;\Theta_2(\mu,b_n) \overset{s}\rightarrow 1$ in $L^2$ as $\mu \uparrow \infty$ uniformly in $n.$ Therefore, by the Trotter approximation theorem  {\small $s$-$L^2$-}$\lim_n e^{-t \Lambda_2(b_n)}$ exists and determines a $C_0$-semigroup in $L^2$. It is also clear that this semigroup is holomorphic and $L^\infty$-contractive. 
\end{proof}

\subsection{Comments}
\label{comm_sect}

1.~The fact that  $b:\mathbb R^d \rightarrow \mathbb R^d$ belongs to  $\mathbf{K}^{d+1}_\delta$ or $\mathbf{F}_\delta$ allows us to construct operator realizations of the formal differential operator $-\Delta +b\cdot \nabla$ as (minus) generators of strongly continuous semigroups in $L^p$ for some or all $p \in [1,\infty),$ $C_\infty$ and/or $C_b$, by means of general tools of the standard perturbation theory (e.g.~theorems of Miyadera \cite{Vo} or Phillips \cite{Ph}, respectively).

2.~Concerning the class $ \mathbf{F}^\frac{1}{2}_\delta$ 
 one can not appeal to the standard perturbation theory (in contrast to $\mathbf{K}^{d+1}_\delta$ and $ \mathbf{F}_\delta$) in order to properly characterize the domain of the generator $\Lambda_p(b)$. Indeed, the arguments in  \cite[p.~413-416]{S}  
 (repeated above in the proof of Theorem \ref{thm3}b) say nothing about $\mathcal{W}^{\alpha, p}$-smoothness of $D(\Lambda_p(b))$ for $p \neq 2$. These arguments rely on  \eqref{first_repr_semenov} that implies that $\Theta_2(\zeta,b)$  is invertible. 
The natural analogue of \eqref{first_repr_semenov} in $L^p$ is valid only for considerably smaller  
class of vector fields: $|b| \in L^{d,\infty}$, although leading to a slightly better smoothness result: $D(\Lambda_p(b)) \subset \mathcal W^{1+\frac{1}{p},p}$, cf.~Theorem \ref{thm1}(\textit{iv}).

\smallskip

3.~Theorem \ref{thm3} is a special case of our Theorem \ref{thm1}. Indeed, the constraints on $p$ and $\delta$ in Theorem \ref{thm1} come solely from the estimate on $\|T_p\|_{p \rightarrow p}$. Now, if $b \in \mathbf{F}^{\scriptscriptstyle \frac{1}{2}}_\delta,$ $\delta<1$, then
$$
\|T_2 \|_{2 \rightarrow 2} \leqslant \| H \|_{2 \rightarrow 2}\| H^* \|_{2 \rightarrow 2} \| \nabla (\zeta - \Delta)^{-\frac{1}{2}} \|_{2 \rightarrow 2} \leqslant \| H \|^2_{2 \rightarrow 2} \leqslant \delta<1.
$$
And if $b \in \mathbf{K}^{d+1}_\delta,\;m_d \delta < 1,$ then $\| T_p \|_{p \to p} < 1$ for \textit{all} $p \in [1, \infty),$ so that the interval $\mathcal I \ni p$ transforms into $[1, \infty)$, and a possible causal dependence of the properties of $D(\Lambda_p(b))$ on $\delta$ gets lost. The latter indicates the smallness of $\mathbf{K}^{d+1}_\delta$ as a subclass of $\mathbf{F}^{\scriptscriptstyle \frac{1}{2}}_\delta.$
\smallskip

4.~Both proofs of Theorem \ref{thm1} and Theorem \ref{thm3} are based on 
similar operator-valued functions,
although the arguments involved differ considerably. 

\smallskip

5.~Note that for $b \in \mathbf{K}^{d+1}_\delta$, $u \in D(\Lambda_1(b))$ satisfies $u \in \mathcal W^{2,1}$, for $b \in \mathbf{F}_\delta$ $u \in \mathcal W^{2,2}_{\loc}$ \cite{KS}, while in Theorem \ref{thm1}(\textit{v}) $u \in D(\Lambda_p(b))$, $p \in \mathcal I$, satisfies $u \in \mathcal W^{2,1}_{\loc}$.


%
6.~Let $b:\mathbb R^d \rightarrow \mathbb R^d$, $b \in \mathbf{F}_\delta^{\scriptstyle \frac{1}{2}}$, $m_d\delta<1$. Theorem \ref{thm1}(\textit{i}) and the argument in the proof of Theorem \ref{thm3}a (using the Riesz interpolation theorem) yields a consistent
 family of $C_0$-semigroups $e^{-t\Lambda_p(b)}$ on $L^p$, for all $p \in (\frac{2}{1+\sqrt{1-m_d\delta}},\infty)$.

\smallskip

7.~The author considers the assertion (\textit{iv}) of Theorem \ref{thm1} (the $\mathcal W^{1+\frac{1}{q},\,p}$-smoothness) as the main result of the paper.
Theorem \ref{thm1}, compared to \cite{KS} and Theorem \ref{thm3}a, covers the larger class of vector fields, and at the same time establishes stronger smoothness properties of $D(\Lambda_p(b))$: $D(\Lambda_p(b)) \subset \mathcal W^{1+\frac{1}{q},\,p}$, $p \in \mathcal I$ 
($q>p$), while in \cite{KS} $D(\Lambda_p(b)) \subset W^{1,jp}$, $jp \in (d, 2j/\sqrt{\delta})$, 
and in Theorem \ref{thm3}a $D(\Lambda_p(b)) \subset \mathcal W^{1,p}$, $p \in [1,\infty)$.

8.~The $C_\infty$-theory of operator $-\Delta+b\cdot\nabla$, $b \in \mathbf{F}_\delta^{\scriptscriptstyle \frac{1}{2}}$ (Theorem \ref{thm2}) follows  almost automatically from 
the $L^p$-theory (Theorem \ref{thm1}) (with $p>d-1$), in contrast to 
\cite{KS}, where the $C_\infty$-theory is obtained from the $L^p$-theory by running a specifically tailored Moser-type iterative procedure (see also \cite{Ki}).

\smallskip

9.~To the author's knowledge, Theorem \ref{thm1} and Theorem \ref{thm2} are the first results where $b$ can combine different kinds of singularities, e.g.~$(|x|-1)^{-\beta}$, $\beta<1$, and $|x|^{-1}$ (originally, the main motivation for this work).

\medskip

\textbf{Acknowledgements.~}\textit{I am deeply grateful to Yu.A.~Semenov for many important suggestions, and constant attention throughout this work.}

\section{Proof of Theorem \ref{thm1}}
\label{thm1_proof}

The method of proof.
We start with operator-valued function
\[
\Theta_p(\zeta,b):= (\zeta - \Delta)^{-1} - Q_p (1 + T_p)^{-1} G_p, \quad \zeta \in \mathcal O,
\]
defined in $L^p$ for each $p$ from the interval
\[
\mathcal{I} := \left]\frac{2}{1 + \sqrt{1-m_d \delta}}, \frac{2}{1 - \sqrt{1-m_d \delta}}\right[, \;\; m_d \delta <1,
\]
and step by step prove that, for $n=1,2,\dots$,
\[
\| \Theta_p(\zeta,b_n) \|_{p \to p}, \; \| \Theta_p(\zeta,b) \|_{p \to p} \leqslant c |\zeta|^{-1};
\]
\[
\Theta_p(\zeta,b_n) \text{ is a pseudo-resolvent};
\] 
\[
 \Theta_p(\zeta,b_n) \text{ coincides with the resolvent } R(\zeta, - \Lambda_p(b_n))=(\zeta+\Lambda_p(b_n))^{-1} \text{ on } \mathcal O;
\]
\[\Theta_p(\zeta,b_n) \overset{s}\rightarrow \Theta_p(\zeta, b) \text{ in } L^p \text{ as $n \uparrow \infty$};
\]
\[
\mu \;\Theta_p(\mu, b_n) \overset{s}\rightarrow 1 \text{ as } \mu \uparrow \infty  \text{ in } L^p \text{ uniformly in } n.
\]

All this combined leads to the conclusion: for each $p \in \mathcal{I}$ there is a holomorphic semigroup $e^{-t \Lambda_p(b)}$ in $L^p$ such that the resolvent $R(\zeta, -\Lambda_p(b))$ on $\zeta \in \mathcal O$ has the representation $\Theta_p(\zeta,b);$

$ \Theta_p(\zeta,b)$ can be written as $ABC,$ where $C \in \mathcal{B}(\mathcal{W}^{-\frac{1}{r^\prime}, \;p}, L^p),\; B \in \mathcal{B}(L^p), \;A \in \mathcal{B}(L^p, \mathcal{W}^{1 + \frac{1}{q}, \;p}),$
 $r < p < q$, $r^\prime = r/(r-1)$.

\medskip

Propositions \ref{est_prop}-\ref{lem_20} below constitute the core of the proof of Theorem \ref{thm1}.

\begin{proposition}
\label{est_prop}

Let  $p \in \mathcal I$.

\smallskip

{\rm(\textit{i})} \; For every $1 \leqslant r<p<q \leqslant \infty$ and $\zeta \in \mathcal O~(=\{\zeta \in \mathbb C: \Real\,\zeta \geqslant \kappa_d\lambda\}, \;\lambda=\lambda_\delta)$
define operators on $L^p$
$$
Q_{p}(q) =(\zeta - \Delta)^{-\frac{1}{2q'}}|b|^\frac{1}{p'}, \quad G_{p}(r) = b^\frac{1}{p} \cdot \nabla (\zeta -\Delta)^{-\frac{1}{2}-\frac{1}{2r}}, \quad T_p = b^\frac{1}{p} \cdot \nabla (\zeta -\Delta)^{-1}|b|^{\frac{1}{p'}}.
$$
Then $G_p(r)$ is bounded: $\|G_{p}(r) \|_{p \to p} \leqslant K_{1,r}.$ $Q_{p}(q)$ and $T_p$ are densely defined (on $\mathcal E$),
and for all $ f \in \mathcal{E}$,
$$
\| Q_{p}(q) f \|_p \leqslant K_{2,q} \| f \|_p ,
 $$
\begin{equation}
\label{T_est}
 \| T_p f \|_p \leqslant m_dc_p\delta  \| f \|_p,  \quad  m_dc_p\delta <1, \quad c_p = p p^\prime/4 . 
\end{equation}
Their extensions by continuity we denote again by $Q_{p}(q)$, $T_p.$

\smallskip

{\rm(\textit{ii})} \;Set $G_p =b^{\frac{1}{p}}\cdot\nabla(\zeta-\Delta)^{-1}$, $Q_p =(\zeta-\Delta)^{-1} |b|^{\frac{1}{p'}}$, $P_p =|b|^{\frac{1}{p}}(\zeta-\Delta)^{-1}$. The operator $Q_p$ is densely defined on $\mathcal E$.
There exist constants $C_i$, $i=1,2,3$, such that
\begin{equation}
\label{G_Q_est}
\| G_p \|_{p \to p} \leqslant C_1 |\zeta|^{-\frac{1}{2p^\prime}}, \quad \| P_p \|_{p \to p} \leqslant C_3 |\zeta|^{-\frac{1}{2}-\frac{1}{2p^\prime}}, \quad \|Q_p f \|_{p} \leqslant C_2 |\zeta|^{-\frac{1}{2}-\frac{1}{2p}}\|f\|_p \quad (f \in \mathcal E),  \quad \zeta \in \mathcal O.
\end{equation}
The extension of $Q_p$ by continuity we denote again by $Q_{p}$.

\smallskip

{\rm(\textit{iii})} \; $\|\Theta_{p}(\zeta,b_n)\|_{p \rightarrow p} \leqslant C_p|\zeta|^{-1}$, $\zeta \in \mathcal O$.

\end{proposition}

\begin{remark}
The proof of Proposition \ref{est_prop} uses ideas in \cite{BS}, \cite{LS}, and appeals to the $L^p$-inequalities between operator $(\lambda-\Delta)^{\frac{1}{2}}$ and ``potential'' $|b|$.
\end{remark}

\begin{proof}
\noindent(\textit{i})~ 
Let $r \in (1,\infty)$. Then

\smallskip

(\textbf{a})~$\mu \geqslant \lambda \;\Rightarrow\; \| |b|^\frac{1}{r}(\mu-\Delta)^{-\frac{1}{2}} \|_{r\to r} \leqslant C_{r,\delta} \mu^{-\frac{1}{2 r^\prime}}$, $C_{r,\delta} = (c_r\delta)^\frac{1}{r}$, $c_r=rr'/4$.

\smallskip

Indeed, 
define in $L^2$  $A = (\mu-\Delta)^{\frac{1}{2}}$, $D(A)= W^{1,2}.$ Then $-A$ is a symmetric Markov generator. Therefore (see e.g.~\cite{LS}), for any $r \in (1,\infty)$,
$$
0 \leqslant u \in D(A_r) \Rightarrow v := u^\frac{r}{2} \in D(A^\frac{1}{2}) \text{ and } c_r^{-1} \|A^\frac{1}{2} v \|_2^2 \leqslant \langle A_r u, u^{r-1} \rangle.
$$
Now let $u$ be the solution of $A_r u = |f|, \;f \in L^r.$ Note that $A \geqslant \sqrt{\mu}$ and $\|u\|_r \leqslant \mu^{-\frac{1}{2}}\|f\|_r.$

Since $b \in \mathbf{F}_\delta^{{\scriptscriptstyle \frac{1}{2}}}$, we have
\[
(c_r \delta)^{-1} \| |b|^\frac{1}{2} v \|_2^2 \leqslant \langle A_r u, u^{r-1} \rangle,
\]
and so $\| |b|^\frac{1}{r} u\|_r^r \leqslant c_r\delta\|f\|_r \|u\|_r^{r-1},$ $\| |b|^\frac{1}{r} A^{-1} |f| \|_r^r \leqslant c_r\delta \mu^{-\frac{r-1}{2}} \|f\|_r^r.$ (\textbf{a}) is proved. 

\medskip

(\textbf{b})~$\mu \geqslant \lambda \;\Rightarrow \;\||b|^\frac{1}{r}(\mu-\Delta)^{-\frac{1}{2}}|b|^\frac{1}{r^\prime} f \|_r \leqslant c_r\delta \|f\|_r$, $f \in \mathcal{E}.$

\smallskip
Indeed, let $u$ be the solution of $A u = |b|^\frac{1}{r^\prime}|f|$, $f \in \mathcal{E}.$ Then, obviously,
\[
\| |b|^\frac{1}{r} u\|_r^r \leqslant c_r\delta \|f\|_r \| |b|^\frac{1}{r} u\|_r^{r-1},
\]
or $\| |b|^\frac{1}{r} u\|_r \leqslant c_r\delta \|f\|_r.$ (\textbf{b}) is proved. 

\medskip

(\textbf{c})~$\mu \geqslant \lambda \; \Rightarrow \; \|(\mu -\Delta)^{-\frac{1}{2}} |b|^\frac{1}{r^\prime} f \|_r \leqslant C_{r', \delta}\,\mu^{-\frac{1}{2r}} \|f\|_r,$ $f \in \mathcal{E}$. 

\smallskip
Indeed, (\textbf{c}) follows from (\textbf{a}) by duality. 

\medskip

Let us prove \eqref{T_est}. 
Let $\zeta \in \mathcal O$. Using \eqref{lem_lambda_est} +   (\textbf{b}) with $r=p \in \mathcal I$, $\mu=\lambda$, we obtain:
$$\| T_p f\|_p \leqslant m_d\|b^{\frac{1}{p}}(\kappa_d^{-1}\Real \zeta-\Delta)^{-\frac{1}{2}}|b|^{\frac{1}{p'}}|f|\|_p \leqslant m_d c_p\delta  \|f\|_p, \quad f \in \mathcal{E}.$$
$m_d c_p\delta <1$ since $p \in \mathcal I$.

Next, we estimate $\|Q_{p}(q)\|_{p \rightarrow p}$, $\|G_{p}(r)\|_{p \rightarrow p}$.
Let $\Real\,\zeta \geqslant \lambda$, $p<q$. 
We obtain:
\begin{align*}
\|Q_{p}(q)f\|_p & \leqslant ~\|(\Real\,\zeta-\Delta)^{-\frac{1}{2q'}}|b|^{\frac{1}{p'}}|f|\|_p \\
\leqslant &~\|(\lambda-\Delta)^{-\frac{1}{2q'}}|b|^{\frac{1}{p'}}|f|\|_p \\
& \text{(here we are using \eqref{repr_81})}\\
& \leqslant 
k_q\int_0^\infty t^{-1+\frac{1}{2q}}\|(t+\lambda-\Delta)^{-\frac{1}{2}}|b|^{\frac{1}{p'}}|f|\|_pdt \\
& \text{(here we are using (\textbf{c}) with $r=p \in \mathcal I$, $\mu=t+\lambda$)} \\
& \leqslant ~
k_qC_{p',\delta} \int_0^\infty t^{-1+\frac{1}{2q}}(t+\lambda)^{-\frac{1}{2p}}dt \;\|f\|_p =K_{2,q}\|f\|_p,\quad f \in \mathcal E.
\end{align*}

Let $\zeta \in \mathcal O$, $1 \leqslant r < p$. Using \eqref{lem_lambda_est_4},
we obtain:
\begin{align*}
\|G_{p}(r)f\|_p &\leqslant m_{r,d}\||b|^{\frac{1}{p}}(\kappa_d^{-1}\Real\,\zeta-\Delta)^{-\frac{1}{2r}} |f|\|_p \\&\leqslant m_{r,d}\||b|^{\frac{1}{p}}(\lambda-\Delta)^{-\frac{1}{2r}} |f|\|_p \qquad\\
& (\text{here we are using \eqref{repr_81} with $q':=r$}) \\ &\leqslant 
m_{r,d}k_{r'}\int_0^\infty t^{-1+\frac{1}{2r'}}\||b|^{\frac{1}{p}}(t+\lambda-\Delta)^{-\frac{1}{2}}|f|\|_pdt \;\;\\
&(\text{here we are using (\textbf{a}) with $r=p \in \mathcal I$, $\mu=t+\lambda$})\\ 
&\leqslant 
m_{r,d}k_{r'}C_{p,\delta} \int_0^\infty t^{-1+\frac{1}{2r'}}(t+\lambda)^{-\frac{1}{2p'}}dt\;\|f\|_p = K_{1,r}\|f\|_p,  \quad f \in \mathcal E.
\end{align*}
The proof of (\textit{i}) is completed.

\medskip

(\textit{ii})~ 
Let $\Real\,\zeta \geqslant \lambda$.
Using \eqref{zeta_est_0} + (\textbf{c}) with $r=p \in \mathcal I$, $\mu=|\zeta|$, we obtain:
\begin{align*}
\|Q_p(2\zeta,b) f\|_{p} &\leqslant  
C_{p^\prime, \delta} 2^{-\frac{d}{4}+\frac{1}{4}}
|\zeta|^{-\frac{1}{2p}}\|f\|_p, \quad f \in \mathcal{E}.
\end{align*}
Now, using the identity
$(\zeta-\Delta)^{-1} = \bigl(1 + \zeta(\zeta-\Delta)^{-1}\bigr) (2\zeta-\Delta)^{-1}$, we obtain:
\begin{align*}
\|Q_p(\zeta,b) f\|_{p} &\leqslant \|1 + \zeta(\zeta-\Delta)^{-1}\|_{p \rightarrow p} \|Q_p(2\zeta,b)f\|_p \\ &\leqslant 2 C_{p^\prime, \delta} 2^{-\frac{d}{4}+\frac{1}{4}}|\zeta|^{-\frac{1}{2p}}
\|f|_p \\
&=C_2|\zeta|^{-\frac{1}{2p}}\|f\|_p, \quad f \in \mathcal E.
\end{align*}

Let $\Real\,\zeta \geqslant \lambda$.
Using \eqref{zeta_est_0}, we obtain:
\begin{align*}
\|P_p(2\zeta,b)\|_{p \rightarrow p} &\leqslant  \||b|^{\frac{1}{p}}(2\zeta-\Delta)^{-\frac{1}{2}}\|_{p \rightarrow p} \|(2\zeta-\Delta)^{-\frac{1}{2}}\|_{p \rightarrow p} \\
&\leqslant 2^{\frac{d}{4}+\frac{1}{4}}\||b|^{\frac{1}{p}}(|\zeta|-\Delta)^{-\frac{1}{2}}\|_{p \rightarrow p} (2|\zeta|)^{-\frac{1}{2}} \qquad \\
&(\text{here we are using (\textbf{a}) with $r=p \in \mathcal I$, $\mu=|\zeta|$}) \\
&\leqslant C_{p,\delta} 2^{\frac{d}{4}+\frac{1}{4}} 2^{-\frac{1}{2}}|\zeta|^{-\frac{1}{2}-\frac{1}{2p'}}.
\end{align*}
Now, using the identity
$(\zeta-\Delta)^{-1} = (2\zeta-\Delta)^{-1}\bigl(1 - \zeta(\zeta-\Delta)^{-1}\bigr)$, we obtain:
\begin{align*}
\|P_p(\zeta,b)\|_{p \rightarrow p} &\leqslant 2 C_{p,\delta} 2^{\frac{d}{4}+\frac{1}{4}} 2^{-\frac{1}{2}}|\zeta|^{-\frac{1}{2}-\frac{1}{2p'}} \\
&=C_3|\zeta|^{-\frac{1}{2}-\frac{1}{2p'}}.
\end{align*}

Let $\zeta \in \mathcal O.$ Using \eqref{zeta_est} + (\textbf{a}) with $r=p \in \mathcal I$, $\mu=|\zeta|$, we obtain:
\begin{align*}
\|G_p(2\kappa_d\zeta,b)  \|_{p \rightarrow p} &\leqslant  m_dC_{p, \delta} 2^{\frac{d}{4}}
|\zeta|^{-\frac{1}{2p'}}. 
\end{align*}
Now, using the identity
$(\zeta-\Delta)^{-1} = (2\kappa_d\zeta-\Delta)^{-1}\bigl(1 - (2\kappa_d-1)\zeta(\zeta-\Delta)^{-1}\bigr)$, we obtain:
\begin{align*}
\|G_p(\zeta,b)  \|_{p \rightarrow p} &\leqslant  2\kappa_d m_dC_{p, \delta} 2^{\frac{d}{4}}
|\zeta|^{-\frac{1}{2p'}} \\
&=C_1|\zeta|^{-\frac{1}{2p'}}. 
\end{align*}
The proof of (\textit{ii}) is completed.

\medskip

(\textit{iii})~ 
By the definition of $\Theta_p(\zeta,b)$, see \eqref{R_p}, for every $\zeta \in \mathcal O$,
\begin{align*}
\|\Theta_p(\zeta,b)\|_{p \rightarrow p}  &\leqslant \|(\zeta-\Delta)^{-1}\|_{p \rightarrow p} + \|Q_p\|_{p \rightarrow p}\|(1+T_p)^{-1}\|_{p \rightarrow p}\|G_p\|_{p \rightarrow p} \\
&\text{(here we are using \eqref{T_est}, \eqref{G_Q_est})}\\
&\leqslant 
|\zeta|^{-1}+C_2 |\zeta|^{-\frac{1}{2}-\frac{1}{2p'}} (1- m_d c_p\delta )^{-1} C_1 |\zeta|^{-\frac{1}{2p}}
  \\
& \leqslant C_p|\zeta|^{-1}, \quad C_p:=1+C_1C_2(1- m_d c_p\delta )^{-1}.
\end{align*}
\end{proof}

\medskip

\begin{remark}
Since $|b_n| \leqslant |b|$ a.e., Proposition \ref{est_prop} is valid for $b_n$, $n=1,2,\dots$, with the same constants.
\end{remark}

\medskip

\begin{proposition}
\label{lem_15}
For every $p \in \mathcal I$, and $n=1,2,\dots$, the operator-valued function $\Theta_{p}(\zeta,b_n)$ is a pseudo-resolvent on $\mathcal O$, i.e.
 \begin{equation*}
\Theta_p(\zeta,b_n) - \Theta_p(\eta,b_n) = (\eta - \zeta) \Theta_p(\zeta,b_n)\Theta_p(\eta,b_n), \quad \zeta,\eta \in \mathcal O.
\end{equation*}
\end{proposition}

\begin{proof}
Define
$
S_\zeta^k:=(-1)^k (\zeta-\Delta)^{-1} b_n  \cdot \nabla (\zeta-\Delta)^{-1} \dots b_n\cdot\nabla (\zeta-\Delta)^{-1}
$, $k:=\#~\text{}b_n$'s. Obviously,
\begin{align*}
\Theta_p(\zeta,b_n)
&:= (\zeta-\Delta)^{-1} - Q\, \bigl(1+T \bigr)^{-1} G 
\\ &= (\zeta-\Delta)^{-1} - Q \,\sum_{k=0}^\infty (-1)^{k}T^{k} \, G = \sum_{k=0}^\infty S_\zeta^k \quad (\text{absolutely convergent in $L^p$}),
\end{align*}
\begin{equation}
\label{RR}
\Theta_p(\zeta,b_n)\Theta_p(\eta,b_n) = \sum_{\ell=0}^\infty \sum_{i=0}^\ell S_\zeta^i S_\eta^{\ell-i}, \quad \zeta,\eta \in \mathcal O.
\end{equation}
Define
\begin{align*}
I^k_{j,m}(\zeta,\eta):=&~(\zeta-\Delta)^{-1} b_n  \cdot\nabla(\zeta-\Delta)^{-1} \dots b_n   \cdot\nabla(\zeta-\Delta)^{-1} 
\\ 
& b_n  \cdot\nabla (\eta-\Delta)^{-1} b_n \cdot\nabla (\eta-\Delta)^{-1}  \dots b_n  \cdot\nabla(\eta-\Delta)^{-1}, \\ & \qquad \quad j:=\#\zeta{\text'}s, \quad m:=\# \eta{\text'}s, \quad k:=\#b_n{\text'}s.
\end{align*}

Substituting the identity $(\zeta-\Delta)^{-1} (\eta-\Delta)^{-1} = (\eta-\zeta)^{-1}\bigl((\zeta-\Delta)^{-1}-(\eta-\Delta)^{-1} \bigr)$ inside the product
\begin{multline*}
S_\zeta^k S_\eta^{j}  \\ = (-1)^{k+j} (\zeta-\Delta)^{-1} b_n  \cdot\nabla (\zeta-\Delta)^{-1} \dots~ b_n\cdot\nabla  \underbrace{ (\zeta-\Delta)^{-1} (\eta-\Delta)^{-1}}_{\scriptscriptstyle (\eta-\zeta)^{-1}((\zeta-\Delta)^{-1}-(\eta-\Delta)^{-1})}  b_n  \cdot\nabla (\eta-\Delta)^{-1} \dots~ b_n \cdot\nabla (\eta-\Delta)^{-1},
\end{multline*}
we obtain
$S_\zeta^k S_\eta^{j} = (\eta-\zeta)^{-1} (-1)^{k+j} \bigl[ I^{k+j}_{k+1,j} - I^{k+j}_{k,j+1} \bigr].$
Therefore,
\begin{multline*}
\sum_{i=0}^\ell S_\zeta^i S_\eta^{\ell-i} =  (\eta-\zeta)^{-1}(-1)^{\ell}\biggl[I^\ell_{1,\ell} - I^\ell_{0,\ell+1} + I^\ell_{2,\ell-1} - I^\ell_{1,\ell} + \dots + I^\ell_{\ell+1,0} - I^{\ell}_{\ell,1} \biggr] \\ =  (\eta-\zeta)^{-1}(-1)^{\ell} \bigl(I_{\ell+1,0}^\ell  - I_{0,\ell+1}^\ell \bigr).
\end{multline*}
Substituting the last identity in the right-hand side of \eqref{RR},
 we obtain
$$
\Theta_p(\zeta,b_n)\Theta(\eta,b_n) = (\eta-\zeta)^{-1}\sum_{\ell=0}^\infty (-1)^{\ell} \bigl(I_{\ell+1,0}^\ell  - I_{0,\ell+1}^\ell \bigr) = (\eta-\zeta)^{-1} \bigl(\Theta_p(\zeta,b_n) h-\Theta_{p}(\eta,b_n)\bigr).
$$
\end{proof}


\begin{proposition}
\label{lem_18}
For every $p \in \mathcal I$, and $n=1,2,\dots$,

\smallskip
{\rm (i)} $\mu \Theta_{p}(\mu,b_n) \overset{s}{\rightarrow} 1 \text{ in $L^p$ as } \mu\uparrow \infty,$

\smallskip

{\rm(ii)} $\|\Theta_{p}(\zeta,b_n)\|_{p \rightarrow p} \leqslant C_p|\zeta|^{-1}$, $\zeta \in \mathcal O$, 
for a constant $C_p$ independent of $n$.

\end{proposition}

\begin{proof}
Proof of (i). Put $\Theta_p \equiv \Theta_p(\mu,b_n)$, $Q_p \equiv Q_p(\mu,b_n)$, $T_p \equiv T_p(\mu,b_n)$, $P_p \equiv P_p(\mu,b_n).$ Since $ \mu(\mu-\Delta)^{-1} \overset{s}{\rightarrow} 1$,
it suffices to show that $\mu \Theta_{p} - \mu(\mu-\Delta)^{-1} \overset{s}{\rightarrow} 0$ in $L^p$.
Since  $\Theta_p\in \mathcal B(L^p)$, and $C_c^\infty$ is dense in $L^p$, it suffices to show that
$\mu \Theta_{p}h - \mu(\mu-\Delta)^{-1}h \rightarrow 0$ in $L^p$ for every $h \in C_c^\infty.$
Write
\begin{equation*}
\Theta_{p}h - (\mu-\Delta)^{-1}h=
-Q_{p} \bigl(1+T_{p} \bigr)^{-1} \,
 b_n^{\frac{1}{p}} \cdot (\mu-\Delta)^{-1}\nabla h.
\end{equation*}
$$
\text{By \eqref{T_est}, } \quad \|\bigl(1+T_{p} \bigr)^{-1}\|_{p \rightarrow p} \leqslant \frac{1}{1-\|T_{p}\|_{p \rightarrow p}} \leqslant \frac{1}{1- m_d c_p\delta }<\infty, \quad \text{by \eqref{G_Q_est},} \quad \|Q_p\|_{p \rightarrow p} \leqslant C_2\mu^{-\frac{1}{2}-\frac{1}{2p}}.$$ 
Again, by \eqref{G_Q_est},
 \begin{align*}
\|b_n^{\frac{1}{p}} \cdot (\mu-\Delta)^{-1}\nabla h\|_p \leqslant &  
\||b_n|^{\frac{1}{p}}(\mu-\Delta)^{-1}  |\nabla h|\|_{p} \\ \leqslant &\, 
\|
P_p
\|_{p \rightarrow p} \|\nabla h\|_{p} \\  
 \leqslant \,&C_3 \mu^{-\frac{1}{2}-\frac{1}{2p'}} \|\nabla h\|_p.
\end{align*}
Therefore,
\begin{align*}
\|\Theta_{p}h - (\mu-\Delta)^{-1}h\|_p  &\leqslant 
\|Q_{p}\|_{p \rightarrow p}   
\|\bigl(1+T_{p} \bigr)^{-1}\|_{p \rightarrow p} \|b_n^{\frac{1}{p}}  \cdot(\mu-\Delta)^{-1} \nabla h\|_p \\ 
&\leqslant 
C_0 \mu^{-\frac{3}{2}}  \|\nabla h\|_p,
\end{align*}
which clearly implies (i).

Proof of (ii). This is Proposition \ref{est_prop}(\textit{iii}).
\end{proof}

\begin{proposition}
\label{lem_20}
For every $p \in \mathcal I$, and $n=1,2,\dots$, we have $\mathcal O \subset \rho(-\Lambda_p(b_n))$, the resolvent set of $-\Lambda_p(b_n)$. The operator-valued function $\Theta_{p}(\zeta,b_n)$ is the resolvent of $-\Lambda_p(b_n)$:
\begin{equation*}
\Theta_{p}(\zeta,b_n)=(\zeta+\Lambda_p(b_n))^{-1}, \quad \zeta \in \mathcal O,
\end{equation*}
and
\begin{equation*}
\|(\zeta+\Lambda_p(b_n))^{-1}\|_{p \rightarrow p} \leqslant C_p|\zeta|^{-1}, \quad \zeta \in \mathcal O.
\end{equation*}

\end{proposition}
\begin{proof}
By definition, we need to verify that, for every $\zeta \in \mathcal O$, $\Theta_p(\zeta,b_n)$ has dense image, and is the left and the right inverse of $\zeta+\Lambda_p(b_n)$. 
Indeed, Proposition \ref{lem_18}(i) implies that $\Theta_p(\zeta,b_n)$ has dense image.
$\Lambda_{p}(b_n):=-\Delta + b_n\cdot \nabla$, $D(\Lambda_p(b_n))=W^{2,p},$ is the generator of a $C_0$-semigroup $e^{-t\Lambda_p(b_n)}$ on $L^p$. 
Clearly,
$\Theta_p(\zeta_n,b_n)=(\zeta_n+\Lambda_p(b_n))^{-1}$ for all sufficiently large $\zeta_n~(=\zeta(\|b_n\|_\infty))$,
therefore, by Proposition \ref{lem_15},
$$
\Theta_p(\zeta,b_n) = (\zeta_n+\Lambda_p(b_n))^{-1}\bigl(1 + (\zeta_n - \zeta) \Theta_p(\zeta,b_n)\bigr), \quad \zeta \in \mathcal O,
$$
so $\Theta_p(\zeta,b_n)L^p \subset D(\Lambda_p(b_n))=W^{2,p}$, and $(\zeta+\Lambda_p(b_n))\Theta_p(\zeta,b_n)g=g$, $g \in L^p$, i.e.~$\Theta_p(\zeta,b_n)$ is the right inverse of $\zeta + \Lambda_p(b_n)$ on $\mathcal O$. Similarly, it is seen that
$\Theta(\zeta,b_n)$ is the left inverse of $\zeta + \Lambda_p(b_n)$ on $\mathcal O$.

\begin{remark}
Alternatively, we could verify conditions of the Kato theorem \cite{K4}: in the reflexive space $L^p$,  the pseudo-resolvent $\Theta_p(\zeta,b_n)$ (see Proposition \ref{lem_15}) 
satisfying
$\mu \Theta_{p}(\mu,b_n) \overset{s}{\rightarrow} 1$ in $L^p$ as $\mu\uparrow \infty$ (see Proposition \ref{lem_18}(ii))
is the resolvent of a densely defined closed operator on $L^p$. This operator coincides with $-\Lambda_p(b_n)$ (since $\Theta_p(\zeta_n,b_n)=(\zeta_n+\Lambda_p(b_n))^{-1}$ for all large $\zeta_n$).
\end{remark}

Now, $\|(\zeta+\Lambda_p(b_n))^{-1}\|_{p \rightarrow p} \leqslant C_p|\zeta|^{-1}$, $\zeta \in \mathcal O,$ follows from Proposition \ref{lem_18}(ii).
\end{proof}

\begin{proposition}
\label{lem_30}
For every $\zeta \in \mathcal O$ and $p \in \mathcal I$,
\begin{equation*}
\Theta_{p} (\zeta,b_n) \overset{s}{\rightarrow} \Theta_{p} (\zeta,b) \text{ in $L^p$},
\end{equation*}
\end{proposition}

\begin{proof}
Put $\Theta_p(b) \equiv \Theta_p(\zeta,b)$, $Q_p(b) \equiv Q_p(\zeta,b)$, $T_p(b) \equiv T_p(\zeta,b)$, $G_p(b) \equiv G_p(\zeta,b)$ (similarly for $b_n$'s). 
It suffices to prove that
\[
Q(b_n)(1+T(b_n))^{-1}G(b_n) \overset{s}\rightarrow Q(b)(1+T(b))^{-1}G(b).
\] 
Thus it suffices to prove consecutively that
\[
G(b_n) \overset{s}\rightarrow G(b), \; (1+T(b_n))^{-1} \overset{s}\rightarrow (1+T(b))^{-1}, \; Q(b_n) \overset{s}\rightarrow Q(b).
\]
In turn, since $(1+T(b_n))^{-1}-(1+T(b))^{-1}= (1+T(b_n))^{-1}(T(b)-T(b_n))(1+T(b))^{-1},$ it suffices to prove that $T(b_n)\overset{s}\rightarrow T(b).$ Finally,
\[
T(b_n)-T(b)= T(b_n)- b_n^{\frac{1}{p}} \cdot\nabla (\zeta-\Delta)^{-1}|b|^{\frac{1}{p'}}+ b_n^{\frac{1}{p}} \cdot\nabla (\zeta-\Delta)^{-1}|b|^{\frac{1}{p'}} - T(b),
\]
and hence we have to prove that
\[
 b_n^{\frac{1}{p}} \cdot \nabla  (\zeta-\Delta)^{-1}|b|^{\frac{1}{p'}} - T(b):=J_n^{(1)} \overset{s}\rightarrow  0 \text{ and } T(b_n)- b_n^s \cdot \nabla (\zeta-\Delta)^{-1}|b|^{\frac{1}{p'}}:= J_n^{(2)} \overset{s}\rightarrow 0.
 \]
 Now, by the Dominated Convergence Theorem (cf.~the argument in the proof of \eqref{A0}), $G(b_n) \overset{s}\rightarrow G(b),$ $J_n^{(1)}|_\mathcal{E} \overset{s}\rightarrow 0.$ Also
 \begin{align*}
\| J_n^{(2)}f \|_p= & \|G(b_n)(|b_n|^{\frac{1}{p'}}-|b|^{\frac{1}{p'}})f \|_p \\
& \leqslant \|G(b_n)\|_{p \to p} \|(|b_n|^{\frac{1}{p'}}-|b|^{\frac{1}{p'}})f \|_p \\
& \leqslant m_d(1+\delta)|\zeta|^{-\frac{1}{2p'}} \|(|b_n|^{\frac{1}{p'}}-|b|^{\frac{1}{p'}})f \|_p, \quad (f \in \mathcal{E}).
\end{align*}
Thus, $ J_n^{(2)}|_ \mathcal{E} \overset{s}\rightarrow 0.$ Since $\|J_n^{(2)}\|_{p \to p}, \|J_n^{(1)}\|_{p \to p} \leqslant m_d \delta,$ we conclude that $T(b_n)\overset{s}\rightarrow T(b).$
It is clear now that $ Q(b_n) \overset{s}\rightarrow Q(b)$.
\end{proof}

Now we are going to prove Theorem \ref{thm1} using the Trotter approximation theorem \cite[IX.2.6]{K}. Recall its conditions (in terms of $\Theta_{p}(\zeta,b_n)$ on the base of Proposition \ref{lem_20}):

1) $\sup_{n \geqslant 1}\|\Theta_{p}(\zeta,b_n)\|_{p \rightarrow p} \leqslant C_p|\zeta|^{-1}$, $\zeta \in \mathcal O$. 

2) 
$\mu \Theta_{p}(\mu,b_n) \overset{s}{\rightarrow} 1 \text{ in $L^p$ as } \mu\uparrow \infty \text{ uniformly in $n$}.$

3) There exists ${\text{\small $s$-$L^p$-}}\lim_n\Theta_{p} (\zeta,b_n)$ for some $\zeta \in \mathcal O$.

Now, 1) is the content of Proposition  \ref{lem_18}(ii). The proof of Proposition \ref{lem_18}(i), in fact, yields 2). Proposition \ref{lem_30} implies 3).

Therefore, by the Trotter approximation theorem, 
  $\Theta_p(\zeta,b)=(\zeta+\Lambda_p(b))^{-1}$, $\zeta \in \mathcal O,$
where $\Lambda_p(b)$ is the generator of the holomorphic $C_0$-semigroup $e^{-t\Lambda_p(b)}$ on $L^p$. 

Hence, the assertions (\textit{i}), (\textit{vi}) of Theorem \ref{thm1} follow.
(\textit{ii}) follows from Proposition  \ref{lem_18}(ii) and Proposition \ref{lem_30}.
 (\textit{iii}) is obvious from the definitions of the operators involved, cf.~Proposition \ref{est_prop}.  (\textit{iii}) $\Rightarrow$ (\textit{iv}).

(\textit{v}) Let $ \zeta \in \mathcal O$. 
By Proposition \ref{lem_30}, $\Lambda_p(b_n)(\zeta+\Lambda_p(b_n))^{-1} \overset{s}{\rightarrow} \Lambda_p(b)(\zeta+\Lambda_p(b))^{-1}$ in $L^p$. Put $Q_p(b) \equiv Q_p(\zeta,b)$, $T_p(b) \equiv T_p(\zeta,b)$, $G_p(b) \equiv G_p(\zeta,b)$ (similarly for $b_n$'s). Since 
$
(\zeta+\Lambda_p(b))^{-1} = (\zeta - \Delta)^{-1} - Q_p(b) (1 + T_p(b))^{-1} G_p(b)
$,
we have
$$b^{\frac{1}{p}}\cdot\nabla (\zeta+\Lambda_p(b))^{-1} = G_p(b) - T_p(b)(1+T_p(b))^{-1}G_p(b)$$ (similarly for $b_n$'s). Since $G_p(b_n) \overset{s}{\rightarrow} G_p(b)$, $T_p(b_n) \overset{s}{\rightarrow} T_p(b)$ in $L^p$ (see the proof of Proposition \ref{lem_30}),
\begin{equation}
\label{conv_5}
\tag{$\ast \ast$}
b_n^{\frac{1}{p}}\cdot\nabla (\zeta+\Lambda_p(b_n))^{-1}  \overset{s}{\rightarrow} b^{\frac{1}{p}}\cdot\nabla (\zeta+\Lambda_p(b))^{-1} \text{ in } L^p. 
\end{equation}
Now, given $u \in D(\Lambda_p(b))$, we have
$u=(\zeta+\Lambda_p(b))^{-1}g$ for some $g \in L^p$, and so, for every $v\in C_c^\infty$,
\begin{align*}
\langle \Lambda_p(b)u,v \rangle &= \langle \Lambda_p(b)(\zeta+\Lambda_p(b))^{-1}g,v \rangle \\ 
&= \lim_n \langle \Lambda_p(b_n)(\zeta+\Lambda_p(b_n))^{-1}g,v \rangle \\ &= \lim_n \langle (\zeta+\Lambda_p(b_n))^{-1}g,-\Delta v\rangle + \lim_n  \langle b_n^{\frac{1}{p}}\cdot\nabla (\zeta+\Lambda_p(b_n))^{-1}g,|{b}_n|^{\frac{1}{p'}} v\rangle \\
& (\text{here we are using \eqref{conv_5} and the fact that $|{b}_n|^{\frac{1}{p'}} v \rightarrow |b|^{\frac{1}{p'}}\, v$ in $L^{p'}$}) \\
&=\langle  (\zeta+\Lambda_p(b))^{-1}g,-\Delta v\rangle + \langle b^{\frac{1}{p}}\cdot\nabla (\zeta+\Lambda_p(b))^{-1}g,|b|^{\frac{1}{p'}}\, v\rangle \\
&=\langle  u ,-\Delta v\rangle +\langle b^{\frac{1}{p}}\cdot\nabla u , |b|^{\frac{1}{p'}}v\rangle.
\end{align*}
Next, by the definition of class $\mathbf{F}_\delta^{\scriptscriptstyle \frac{1}{2}}$, $|b| \in L^1_{\loc}$.
Since for $u \in D(\Lambda_p(b))$, $b^{\frac{1}{p}}\cdot\nabla u \in L^p$, it follows that  $b\cdot \nabla u=|b|^{\frac{1}{p'}} \; b^{\frac{1}{p}}\cdot\nabla u \in L^1_{\loc}$. Also, $\Lambda_p(b)u \in L^p$, and hence $\langle \Lambda_p(b)u,v \rangle  = \langle u ,-\Delta v\rangle +\langle b\cdot \nabla u , v\rangle$. Therefore, $\Delta u \text{ (understood in the sense of distributions)}= -\Lambda_p(b)u + b \cdot \nabla u \in L^1_{\loc},
$
i.e.~$u \in \mathcal W^{2,1}_{\loc}$. The proof of (\textit{v}) is completed.

For the proof of (\textit{viii}) see the argument in \cite[p.~415-416]{S}.

The proof of Theorem \ref{thm1} is completed.

\section{Proof of Theorem \ref{thm2}}


It is easily seen that, due to the strict inequality $m_d \delta<4\frac{d-2}{(d-1)^2}$, 
for every $\tilde{\delta}>\delta$ such that $m_d \tilde{\delta}<4\frac{d-2}{(d-1)^2}$ there  exist $\varepsilon_n>0$, $\varepsilon_n\downarrow 0$, such that 
$$
\tilde{b}_{n}:=\eta_{\varepsilon_n} \ast b_n \in \mathbf{F}_{\tilde{\delta}}^{\scriptscriptstyle \frac{1}{2}}, \quad n=1,2,\dots
$$

\noindent (\textit{i}) We verify conditions of the Trotter approximation theorem:
\smallskip

$1^{\circ}$) $\sup_n\|(\mu+\Lambda_{C_\infty}(\tilde{b}_n))^{-1}\|_{\infty \rightarrow \infty} \leqslant 1$, $\mu \geqslant \kappa_d \lambda$.

$2^{\circ}$) $\mu (\mu+\Lambda_{C_\infty}(\tilde{b}_n))^{-1} \rightarrow 1$ in $C_\infty$ as $\mu \uparrow \infty$ uniformly in $n$.

$3^{\circ}$) There exists $\text{\small $s\text{-}C_\infty\text{-}$}\lim_n (\mu+\Lambda_{C_\infty}(\tilde{b}_n))^{-1}$ for some $\mu \geqslant \kappa_d \lambda$.

\smallskip

The condition $1^{\circ}$) is immediate. In view of $1^{\circ}$), it suffices to verify $2^{\circ}$), $3^{\circ}$) 
on $\mathcal S$, the  L.~Schwartz space of test functions. Fix $p \in \mathcal I$, $p>d-1$ (such $p$ exists since $m_d \tilde{\delta}<4\frac{d-2}{(d-1)^2}$). 


\begin{proposition}
\label{prop_id}
For every $\mu \geqslant \kappa_d\lambda$, $n=1,2,\dots$,
$
\Theta_p(\mu,\tilde{b}_n)\mathcal S\subset \mathcal S,
$ and
$$
(\mu+\Lambda_{C_\infty}(\tilde{b}_n))^{-1}|_{\mathcal S} = \Theta_p(\mu,\tilde{b}_n)|_{\mathcal S}.
$$
\end{proposition}
\begin{proof} 
The inclusion $\Theta_p(\mu,\tilde{b}_n)\mathcal S\subset \mathcal S$ is clear.
Clearly,
$\Theta_p(\mu_n,\tilde{b}_n)|_{\mathcal S}=(\mu_n+\Lambda_{C_\infty}(\tilde{b}_n))^{-1}|_{\mathcal S}$ for all sufficiently large $\mu_n~(=\mu(\|\tilde{b}_n\|_\infty))$.
By $\Theta_p(\mu,\tilde{b}_n)\mathcal S\subset \mathcal S$ and Proposition \ref{lem_15}, $\Theta_p(\mu,\tilde{b}_n)|_{\mathcal S}$ satisfies the resolvent identity on $\mu \geqslant \kappa_d\lambda$, 
$$
\Theta_p(\mu,\tilde{b}_n)|_{\mathcal S} = (\mu_n+\Lambda_{C_\infty}(\tilde{b}_n))^{-1}\bigl(1 + (\mu_n - \mu) \Theta_p(\mu,\tilde{b}_n)\bigr)|_{\mathcal S}, \quad \mu \geqslant \kappa_d\lambda,
$$
so $\Theta_p(\mu,\tilde{b}_n)|_{\mathcal S}$ is the right inverse of $\mu + \Lambda_{C_\infty}(\tilde{b}_n)|_{\mathcal S}$ on $\mu \geqslant \kappa_d\lambda$. Similarly, it is seen that
$\Theta_p(\mu,\tilde{b}_n)|_{\mathcal S}$ is the left inverse of $\mu + \Lambda_{C_\infty}(\tilde{b}_n)|_{\mathcal S}$ on $\mu \geqslant \kappa_d\lambda$.
\end{proof}

\begin{proposition}
\label{prop_1}
For every $\mu \geqslant \kappa_d\lambda$,
$\Theta_p(\mu,b)\mathcal S \subset C_\infty,$ and
$$
\Theta_p(\mu,\tilde{b}_n) \overset{s}{\rightarrow} \Theta_p(\mu,b) \text{ in $C_\infty$}.
$$
\end{proposition}
\begin{proof}
By Theorem \ref{thm1}(\textit{iv}), since $p>d-1$, 
$\Theta_p(\mu,b) L^p \subset C_\infty$. 
Put $$Q_p(q,b) \equiv Q_p(q,\mu,b), \quad T_p(b) \equiv T_p(\mu,b), \quad G_p(b) \equiv G_p(\mu,b).$$  
To establish the required convergence, it suffices to prove that
\[
(\mu-\Delta)^{-\frac{1}{2}-\frac{1}{2q}}Q(q,\tilde{b}_n)(1+T(\tilde{b}_n))^{-1}G(\tilde{b}_n) \overset{s}\rightarrow (\mu-\Delta)^{-\frac{1}{2}-\frac{1}{2q}}Q(q,b)(1+T(b))^{-1}G(b) \text{ in $C_\infty$}.
\] 
We choose $q~(>p)$ close to $d-1$ so that $(\mu-\Delta)^{-\frac{1}{2}-\frac{1}{2q}}L^p \hookrightarrow C_\infty$.
Thus it suffices to prove that
\[
G(\tilde{b}_n) \overset{s}\rightarrow G(b), \; (1+T(\tilde{b}_n))^{-1} \overset{s}\rightarrow (1+T(b))^{-1}, \; Q(q,\tilde{b}_n) \overset{s}\rightarrow Q(q,b) \text{ in $L^p$},
\]
which can be done by repeating the arguments in the proof of Proposition \ref{lem_30}.
\end{proof}

\begin{proposition}
\label{lem_4}
\begin{equation}
\label{star5}
\mu \Theta_{p}(\mu,\tilde{b}_n) \overset{s}{\rightarrow} 1 \text{  as $\mu \uparrow \infty$ in $C_\infty$ uniformly in $n$}.
\end{equation}
\end{proposition}

\begin{proof}
Put $\Theta_p \equiv \Theta_p(\mu,\tilde{b}_n)$, $T_p \equiv T_p(\mu,\tilde{b}_n)$.
Since $\mu(\mu-\Delta)^{-1} \overset{s}{\rightarrow}1$ in $C_\infty$, and $\mathcal S$ is dense in $C_\infty$, 
it suffices to show that $\|\mu \Theta_{p}f- \mu(\mu-\Delta)^{-1}f\|_\infty \rightarrow 0
$ for every $f \in \mathcal S$. For each $f \in \mathcal S$ there is $h \in \mathcal S$ such that $f=(\lambda-\Delta)^{-\frac{1}{2}}h$, where $\lambda=\lambda_\delta>0$.
Let $q>p$.
Write
$$
\Theta_{p}f- (\mu-\Delta)^{-1}f =
-(\mu-\Delta)^{-\frac{1}{2}-\frac{1}{2q}} Q_{p}(q)   
\bigl(1+T_{p} \bigr)^{-1} b^{\frac{1}{p}} (\lambda-\Delta)^{-\frac{1}{2}} \cdot (\mu-\Delta)^{-1}\nabla h.
$$
Now, arguing as in the proof of Proposition \ref{lem_18}(ii), but using estimates
$$
\|(\mu-\Delta)^{-\frac{1}{2}-\frac{1}{2q}}\|_{p \rightarrow \infty} \leqslant c \mu^{-\frac{1}{2}+\frac{d}{2p}-\frac{1}{2q}}, \quad c<\infty, \quad \text{ and } \quad \|Q_{p}(q)\|_{p \rightarrow p} \leqslant \tilde{K}_{2,q}<\infty \quad \text{(see \eqref{G_Q_est})},
$$
we obtain
$$
\|\Theta_{p}f - (\mu-\Delta)^{-1}f\|_\infty 
\leqslant C \mu^{-\frac{1}{2}+\frac{d}{2p} -\frac{1}{2q}} \mu^{ -1}  \|\nabla h\|_p.
$$
Since $p>d-1$, choosing $q$ sufficiently close to $p$, we obtain  $$-\frac{1}{2}+\frac{d}{2p}-\frac{1}{2q}-1<-1,$$ so $\mu \Theta_{p} - \mu(\mu-\Delta)^{-1} \overset{s}{\rightarrow} 0$ in $C_\infty$, as needed.
\end{proof}

Now, Proposition \ref{prop_1} verifies condition $3^{\circ}$), and Proposition \ref{lem_4} verifies condition $2^{\circ}$). 
The assertion (\textit{i}) of Theorem \ref{thm2} now follows from the Trotter approximation theorem.

Assertion (\textit{ii}) of Theorem \ref{thm2} follows from Theorem \ref{thm1}(\textit{iii}).

The proof of assertion (\textit{iii})  is standard, and is omitted.

\begin{remark}
We could construct $e^{-t\Lambda_{C_\infty}(b)}$ alternatively as follows:
$$
e^{-t\Lambda_{C_\infty}(b)}:=\bigl(e^{-t\Lambda_p(b)}|_{C_\infty \cap L^p} \bigr)_{C_\infty}^{\clos}
\;\text{ (after a change on a set of measure zero)}, \quad t>0,
$$
where $p \in \bigl(d-1,\frac{2}{1-\sqrt{1-m_d\delta}}\bigr)$.
\end{remark}

\appendix

\section{}

Define $I_n:= \|(b-b_n)\cdot \nabla (\zeta - \Delta)^{-1} f \|_1$.

\begin{estimate}
Let $b \in \mathbf{K}^{d+1}_\delta$. For every $f \in L^1$ and $\Real\,\zeta \geqslant \kappa_d\lambda$,
\begin{equation}
\label{A0}
\tag{A.0}
I_n \to 0 \text{ as } n \uparrow \infty. 
\end{equation}
\end{estimate}

\begin{proof}[Proof of \eqref{A0}]
Since $I_n \leqslant 2 m_d \||b|(\lambda -\Delta)^{-\frac{1}{2}}|f|\|_1 \leqslant 2 m_d \delta \|f\|_1,$ it suffices to prove \eqref{A0} for each $f \in L^1 \cap L^\infty.$ Let $f \in L^1 \cap L^\infty, \lambda > 0$ and $b$ be fixed. Since $|b|(\lambda -\Delta)^{-\frac{1}{2}} |f| \in L^1,$ for a given $\epsilon > 0,$ there exists $\mathcal{K},$ a compact, such that
\[
\|(\mathbf{1}-\mathbf{1}_\mathcal{K})|b|(\lambda -\Delta)^{-\frac{1}{2}} |f| \|_1 \leqslant \epsilon,
\]
where $\mathbf{1}_{\mathcal K}$ is the characteristic function of $\mathcal K$.
Define $I_{\mathcal{K}, n} := \| \mathbf{1}_\mathcal{K} |b - b_n|(\lambda-\Delta)^{-\frac{1}{2}} |f| \|_1.$ Clearly,
\[
I_{\mathcal{K}, n} \leqslant \lambda^{-\frac{1}{2}} \|f\|_\infty \|\mathbf{1}_\mathcal{K} |b - b_n| \|_1.
\]
Since $|b| \in L^1_\loc$ and $\mathcal{K}$ independent of $n=1,2,\dots,$
\[
\|\mathbf{1}_{\mathcal K}|b-b_n|\|_1 \leqslant \|\mathbf{1}_{|b| \geqslant n}(\mathbf{1}_{\mathcal K}|b|)\|_1 \rightarrow 0 \text{ as } n\uparrow \infty.
\]
Therefore, for a given $\epsilon,$ there exists $n_0 = n_0(\epsilon) \geqslant 1,$ such that $I_{\mathcal{K}, n} \leqslant \epsilon$ whenever $n \geqslant n_0,$ and so
\[
I_n \leqslant 3 m_d \epsilon \qquad \forall n \geqslant n_0.
\]
\end{proof}

We use the following pointwise estimates ($x,y \in \mathbb R^d$, $x \neq y$).

\label{lem_lambda_est_proof}

\begin{estimate}
For every $\Real \zeta>0$,
\begin{equation}
\label{lem_lambda_est}
\tag{A.1}
|\nabla (\zeta-\Delta)^{-1}(x,y)| \leqslant m_d  (\kappa_d^{-1}\Real\zeta-\Delta)^{-\frac{1}{2}}(x,y), 
\end{equation}
where $m_d^2 := \pi (2e)^{-1} d^d (d-1)^{1-d},$  $\kappa_d:=\frac{d}{d-1}$.

\medskip

For every $r \in (1,\infty]$
there exists a constant $m_{r,d}<\infty$ such that for all  $\Real \zeta>0$,
\begin{equation}
\label{lem_lambda_est_4}
\tag{A.2}
|\nabla (\zeta-\Delta)^{-1+\frac{1}{2r}}(x,y)| \leqslant m_{r,d} (\kappa_d^{-1}\Real\zeta-\Delta)^{-\frac{1}{2}+\frac{1}{2r}}(x,y). 
\end{equation}
\end{estimate}

\begin{estimate}
For every $\Real \zeta>0$,
\begin{equation}
\label{zeta_est}
\tag{A.3}
|\nabla (\zeta-\Delta)^{-1}(x,y)| \leqslant 
2^{\frac{d}{4}} m_d \biggl(\kappa_d^{-1}
2^{-\frac{1}{2}}|\zeta|
-\Delta\biggr)^{-\frac{1}{2}}(x,y),
\end{equation}
\begin{equation}
\label{zeta_est_0}
\tag{A.4}
|(\zeta-\Delta)^{-\frac{1}{2}}(x,y)| \leqslant 
2^{\frac{d}{4}+\frac{1}{4}} \biggl(
2^{-\frac{1}{2}}|\zeta|
-\Delta\biggr)^{-\frac{1}{2}}(x,y).
\end{equation}
\end{estimate}

\begin{proof}[Proof of \eqref{lem_lambda_est}]
Let $\alpha \in (0,1)$. Set $c(\alpha):=\sup_{\xi>0}\xi e^{-(1-\alpha)\xi^2}~\biggl(=\frac{1}{\sqrt{2}} (1-\alpha)^{-\frac{1}{2}}e^{-\frac{1}{2}}\biggr)$,
so that
\begin{equation}
\tag{$\star$}
\label{req_ineq_99}
\xi e^{-\xi^2} \leqslant c(\alpha)  e^{-\alpha\xi^2} \quad \text{ for all } \xi>0.
\end{equation}
We use the well known formula
\begin{equation*}
(\zeta-\Delta)^{-\frac{\gamma}{2}} (x,y) = \frac{1}{\Gamma\bigl(\frac{\gamma}{2}\bigr)}\int_0^\infty e^{-\zeta t} t^{\frac{\gamma}{2}-1} (4\pi t)^{-\frac{d}{2}}e^{-\frac{|x-y|^2}{4t}} dt, \quad 0<\gamma \leqslant 2,
\end{equation*}
first with $\gamma=2$, and then with $\gamma=1$,
to obtain:
\begin{align*}
|\nabla (\zeta-\Delta)^{-1}(x,y)| 
& \leqslant \int_0^\infty e^{-t\Real\zeta} (4\pi t)^{-\frac{d}{2}} \frac{|x-y|}{2t} e^{-\frac{|x-y|^2}{4t}} dt \\
&\leqslant c(\alpha) \int_0^\infty  e^{-t\Real \zeta} t^{-\frac{1}{2}} (4\pi t)^{-\frac{d}{2}} e^{-\alpha \frac{|x-y|^2}{4t}}dt \qquad \biggl(  \text{By \eqref{req_ineq_99} with $\xi:=\frac{|x-y|}{2\sqrt{t}}$ } \biggr)\\  &\leqslant c(\alpha) \alpha^{-\frac{1}{2}-\frac{d}{2}+1} \int_0^\infty e^{-(\Real \zeta)\alpha t} t^{-\frac{1}{2}} (4\pi t)^{-\frac{d}{2}} e^{-\frac{|x-y|^2}{4t }}  dt \qquad \biggl(\text{change } t/\alpha \text{ to } t \biggr) \\ &= c(\alpha) \alpha^{\frac{1}{2}-\frac{d}{2}} \Gamma\left(\frac{1}{2}\right) \bigl(\alpha\Real \zeta-\Delta \bigr)^{-\frac{1}{2}}(x,y).
\end{align*}
Now, we minimize
$
c(\alpha) \alpha^{\frac{1}{2}-\frac{d}{2}}\Gamma(\frac{1}{2})
$ in $\alpha \in (0,1)$.
The minimum is attained at $\alpha_d=\frac{d-1}{d}~(=:\kappa_d^{-1})$, and is equal to $m_d$.

The proof of \eqref{lem_lambda_est_4} is similar.
\end{proof}

\begin{proof}[Proof of \eqref{zeta_est}]First, suppose that $\Imag~\zeta \leqslant 0$. 
By Cauchy theorem, 
\begin{equation*}
(\zeta-\Delta)^{-1}(x,y)=\int_0^\infty e^{-\zeta t}   (4\pi t)^{-\frac{d}{2}} e^{-\frac{|x-y|^2}{4t}}  dt 
= \int_0^\infty  e^{-\zeta re^{i\frac{\pi}{4}}} e^{-i\frac{\pi}{4}\frac{d}{2}} (4\pi r)^{-\frac{d}{2}} e^{-\frac{|x-y|^2}{4re^{i\frac{\pi}{4}}}} e^{i\frac{\pi}{4}}  dr,  
\end{equation*}
(i.e.~we have changed the contour of integration from $\{t: t \geqslant 0\}$ to $\{re^{i \frac{\pi}{4} }:r \geqslant  0\}$).
Thus,
$$
|\nabla (\zeta-\Delta)^{-1}(x,y)| \leqslant  \int_0^\infty  \left|e^{-\zeta re^{i\frac{\pi}{4}}}\right| (4\pi r)^{-\frac{d}{2}} \left|\frac{x-y}{2r} \right|\biggl| e^{-\frac{|x-y|^2}{4re^{i\frac{\pi}{4}}}}\biggr| dr.
$$
We have 
$$
|e^{-\zeta re^{i\frac{\pi}{4}}}| \leqslant e^{-r\frac{1}{\sqrt{2}}(\Real\,\zeta-\Imag\,\zeta)}, \quad
\bigl|e^{-\frac{|x-y|^2}{4re^{i\frac{\pi}{4}}}} \bigr| \leqslant e^{-\frac{|x-y|^2}{4r}\frac{1}{\sqrt{2}}}, \quad \Real\,\zeta-\Imag\,\zeta \geqslant |\zeta|.
$$
Therefore,
\begin{align*}
|\nabla (\zeta-\Delta)^{-1}(x,y)| \leqslant &
\int_0^\infty  e^{-r\frac{1}{\sqrt{2}}|\zeta|}(4\pi r)^{-\frac{d}{2}} \left|\frac{x-y}{2r} \right| e^{-\frac{|x-y|^2}{4r}\frac{1}{\sqrt{2}}}dr \qquad \bigl(\text{change }r\sqrt{2} \text{ to } r \bigr)  \\ 
=~ & 2^{\frac{d}{4}}\int_0^\infty  e^{-r \frac{1}{2}|\zeta|}(4\pi r)^{-\frac{d}{2}} \left|\frac{x-y}{2r} \right| e^{-\frac{|x-y|^2}{4r}}dr  \\
\leqslant~ &\frac{2^{\frac{d}{4}}m_d}{\Gamma\bigl(\frac{1}{2} \bigr)}\int_0^\infty  e^{-r\kappa_d^{-1}\frac{1}{2}|\zeta|}(4\pi r)^{-\frac{d}{2}} r^{-\frac{1}{2}}  e^{-\frac{|x-y|^2}{4r}}dr \quad \bigl(\text{cf.~proof of \eqref{lem_lambda_est}} \bigr)  \\
=~ & 2^{\frac{d}{4}}m_d \left( \kappa_d^{-1}2^{-1}|\zeta|-\Delta\right)^{-\frac{1}{2}}(x,y) 
\end{align*}
which yields \eqref{zeta_est} for $\Imag\,\zeta \leqslant 0$.
The case $\Imag \zeta >0 $ is treated analogously. 
\end{proof}

\begin{proof}[Proof of \eqref{zeta_est_0}]
First, suppose that $\Imag~\zeta \leqslant 0$. 
By Cauchy theorem, 
\begin{multline*}
(\zeta-\Delta)^{-\frac{1}{2}}(x,y)=\int_0^\infty e^{-\zeta t}  t^{-\frac{1}{2}} (4\pi t)^{-\frac{d}{2}} e^{-\frac{|x-y|^2}{4t}}  dt \\
= \int_0^\infty  e^{-\zeta re^{i\frac{\pi}{4}}} r^{-\frac{1}{2}} e^{-i\frac{\pi}{8}}  e^{-i\frac{\pi}{4}\frac{d}{2}}  (4\pi r)^{-\frac{d}{2}} e^{-\frac{|x-y|^2}{4re^{i\frac{\pi}{4}}}} e^{i\frac{\pi}{4}}  dr,  
\end{multline*}
so we estimate as above:
\begin{align*}
|(\zeta-\Delta)^{-\frac{1}{2}}(x,y)| &\leqslant \int_0^\infty  e^{-r\frac{1}{\sqrt{2}}|\zeta|}r^{-\frac{1}{2}}(4\pi r)^{-\frac{d}{2}} e^{-\frac{|x-y|^2}{4r}\frac{1}{\sqrt{2}}}dr \\
&= 2^{\frac{d}{4}+\frac{1}{4}} \left( 2^{-1}|\zeta|-\Delta\right)^{-\frac{1}{2}}(x,y).
\end{align*}
The case $\Imag \zeta >0 $ is treated analogously. 
\end{proof}

\begin{estimate}
In the proof of Proposition \ref{est_prop} we need the following formula: for every $\Real\,\zeta > 0$, $q \in (1,\infty)$,
\begin{equation}
\label{repr_81}
\tag{A.5}
(\zeta-\Delta)^{-\frac{1}{2q'}} 
=c_q\int_0^\infty t^{-1+\frac{1}{2q}}(t+\zeta-\Delta)^{-\frac{1}{2}}dt, \quad c_q:=\frac{\Gamma\bigl(\frac{1}{2}\bigr)}{\Gamma\bigl(\frac{1}{2q}\bigr)\Gamma\bigl(\frac{1}{2q'}\bigr)},
\end{equation}
\end{estimate}

\end{document}